\documentclass[10pt,reqno]{amsart}
\usepackage{tikz-cd} 
\usepackage[english]{babel}
\usepackage [latin1,utf8]{inputenc}
\usepackage[T1]{fontenc}
\usepackage{amsmath}
\usepackage{amstext}
\usepackage{amsfonts}
\usepackage{amssymb}
\usepackage{amsthm}
\usepackage{mathabx}
\usepackage{soul}
\usepackage{mathtools}
\usepackage[shortlabels]{enumitem}
\usepackage{dsfont}
\usepackage{nicefrac}
\usepackage{color}
\usepackage{graphicx}
\usepackage[colorinlistoftodos]{todonotes}
\usepackage[colorlinks=true, allcolors=blue]{hyperref}
\usepackage{float}
\usepackage[colorinlistoftodos]{todonotes}
\usepackage{theoremref}
\usepackage{enumitem}
\usepackage{dirtytalk}
\allowdisplaybreaks
\newtheorem{theorem}{Theorem}[section]

\newtheorem{prop}[theorem]{Proposition}
\newtheorem{claim}[theorem]{Claim}

\newtheorem{lem}[theorem]{Lemma}

\newtheorem{question}[theorem]{Question}
\newtheorem*{cor*}{Corollary}
\newtheorem*{thm*}{Theorem}
\newtheorem*{lem*}{Lemma}

\newtheorem*{prop*}{Proposition}
\newtheorem{claim*}{Claim}
\theoremstyle{definition}
\newtheorem{definition}[theorem]{Definition}

\newtheorem{example}[theorem]{Example}
\newtheorem*{defn*}{Definition}

\theoremstyle{remark}
\newtheorem{remark}[theorem]{Remark}

\newcommand{\SL}{\operatorname{SL}}
\newcommand{\PSL}{\operatorname{PSL}}
\newcommand{\GL}{\operatorname{GL}}

\newcommand{\Id}{\operatorname{Id}}
\newcommand{\Aut}{\operatorname{Aut}}
\newcommand{\Out}{\operatorname{Out}}
\newcommand{\Prob}{\operatorname{Prob}}
\newcommand{\conv}{\operatorname{conv}}

\newcommand{\act}{\curvearrowright}

\newcommand{\Gr}{\operatorname{Gr}}

\newcommand{\cA}{\mathcal{A}}

\newcommand{\cB}{\mathcal{B}}

\newcommand{\bP}{{\mathbb{P}}}
\newcommand{\bR}{{\mathbb{R}}}

\newcommand{\N}{{\mathbb{N}}}
\newcommand{\Z}{{\mathbb{Z}}}

\newcommand{\R}{{\mathbb{R}}}
\newcommand{\C}{{\mathbb{C}}}

\newcommand{\Homeo}{\operatorname{Homeo}}
\newcommand{\id}{\operatorname{id}}

\newcommand{\E}{\mathbb{E}}

\newcommand{\norm}[1]{\left \lVert #1 \right \rVert}
\newcommand{\abs}[1]{\left \lvert #1 \right \rvert}
\renewcommand{\phi}{\varphi}

\newcommand{\nor}{\vartriangleleft}
\makeatletter \def\l@subsection{\@tocline{2}{0pt}{1pc}{5pc}{}} \def\l@subsection{\@tocline{2}{0pt}{2pc}{6pc}{}} \makeatother

\author{Tattwamasi Amrutam}
\address{Ben Gurion University of the Negev.
	Department of Mathematics.
	Be'er Sheva, 8410501, Israel.
}
\email{tattwamasiamrutam@gmail.com}
\author{Eli Glasner}
\address{Tel-Aviv University.
	Department of Mathematics.
	Tel-Aviv, Israel.
}
\email{glasner@math.tau.ac.il}
\author{Yair Glasner} 
\address{Ben Gurion University of the Negev.
	Department of Mathematics.
	Be'er Sheva, 8410501, Israel.
}
\email{yairgl@math.bgu.ac.il}

\subjclass[2010]{Primary 37A55, 37B05; Secondary 46L55}
\keywords{$C^*$-crossed products, intermediate subalgebras, rigid,
strong proximality, plumps sets, $C^*$-simple groups}

\thanks{This research was supported by grants from the Israel Science Foundation:
ISF 1175/18 for the first author,
ISF 1194/19 for the second, and
ISF 2919/19 for the third. The first named author is also supported by research funding from the European Research Council
(ERC) under the European Union's Seventh Framework Program
(FP7-2007-2013) (Grant agreement No. 101078193).}

\date{\today}

\begin{document}
\title[Rigidity Vs. Strong Proximality]{Crossed products of dynamical systems; rigidity Vs. strong proximality}
\begin{abstract}
Given a dynamical system $(X, \Gamma)$, call
the corresponding crossed product
$C^*$-algebra $C(X) \rtimes_{r} \Gamma$ {\it{reflecting}} if every intermediate
$C^*$-algebra
$C^*_r(\Gamma)  < \cA < C(X) \rtimes_{r} \Gamma$ is of the form
$\mathcal{A} =  C(Y) \rtimes_{r} \Gamma$, corresponding to a dynamical factor $X \rightarrow Y$. It is called {\it{almost reflecting}} if $\E(\cA) \subset \cA$ for every such $\cA$. These two notions coincide for groups admitting the approximation property (AP). Let $\Gamma$ be a non-elementary convergence group or a lattice in $\SL_d(\R)$ for some $d \ge 2$. We show that any uniformly rigid system $(X,\Gamma)$ is almost reflecting. In particular, this holds for any equicontinuous action. In the von Neumann setting, for the same groups $\Gamma$ and any uniformly rigid system $(X,\cB,\mu, \Gamma)$ the crossed product algebra $L^{\infty}(X,\mu) \rtimes \Gamma$ is reflecting. An inclusion of algebras $\cA \subset \cB$ is called {\it{minimal ambient}} if there are no intermediate algebras. As a demonstration of our methods, we construct examples of minimal ambient inclusions with various interesting properties in the $C^*$ and the von Neumann settings. 
\end{abstract}
\maketitle
\tableofcontents
\newpage
\section{Introduction}
Let $\Gamma$ be a discrete group acting by homeomorphisms on a compact space $X$. The pair $(X, \Gamma)$ is called a topological dynamical system or a $\Gamma$-flow. Such a flow is called 
%{\it{rigid}} or 
{\it{uniformly rigid}} if there exists a sequence (or a net in the non-metrizable case) $\{\gamma_i\}$ of elements that are nontrivial but converge uniformly to the identity homeomorphism $\Id_X$. Such a sequence is referred to as {\it{a rigid sequence on $X$}}. Every equicontinuous flow (of an infinite group) is rigid, but 
the class of uniformly rigid flows is much wider.
E.g., it was recently shown in \cite{glasner2022rigid} that 
every residually finite group admits many such flows that are not equicontinuous. With any $\Gamma$-flow, one associates in a standard way a crossed product $C^*$-algebra $C(X) \rtimes_r \Gamma$. 

Similarly, when $\Gamma$ acts by measure preserving transformations on a standard probability space, we obtain a measurable dynamical system $(X,\cB,\mu,\Gamma)$. Such a system is called {\it{rigid}} if there is a net $\{\gamma_i\}$ of elements that act non-trivially but converge to the identity in the sense that $\lim_{n \rightarrow \infty}\mu(\gamma_n A \triangle A) = 0, \ \forall A \in \cB$. Namely $\gamma_n \rightarrow \Id$ in the Polish group $\Aut(X,\cB,\mu)$, endowed with the weak topology (See \cite{Kechris:GA} for details on the Polish topology on this group). With a measurable dynamical system, one associates a crossed product von Neumann algebra $L^{\infty}(X,\mu) \rtimes \Gamma$. Such crossed product algebras combine into a single package the topological or measurable properties of the space together with the action of the group. Our results pertain to both topological and measurable settings. In order to improve the readability, we proceed, throughout the introduction, to discuss the topological setting. 

Adopting the perspective undertaken in \cite{A19,amrutam2023intermediate}, let us consider the intermediate subalgebras 
$$
\{\cA \ | \ C^{*}_{r} (\Gamma) < \cA < C(X) \rtimes_r \Gamma\},
$$
 as non-commutative factors of the dynamical system $(X,\Gamma)$. We are aware of two sources for such intermediate algebras. Firstly, any dynamical $\Gamma$-factor $X \rightarrow Y$ gives rise to dynamical algebras of the form $C(Y) \rtimes_r \Gamma$. Secondly, as in \cite{amrutam2023intermediate} ideals $I \lhd C^*_r(\Gamma)$ often give rise to intermediate algebras. 
\begin{question} \label{q:two_sources}
Is it true that these are essentially the only sources for intermediate algebras?
\end{question}
It is clear from that classification of intermediate algebras for an irrational rotation of the circle $(S^1, T_{\alpha})$ given in the aforementioned paper that ideals in $C^*_r (\Gamma)$ give rise to intermediate algebras in complicated ways, which we are far from understanding entirely. Thus, it makes sense first to address Question \ref{q:two_sources} in the $C^*$-simple case. For this, the following notation plays an important part:
\begin{definition} \label{def:reflecting}
A crossed product $C(X) \rtimes_r \Gamma$ (or in the measurable setting $L^{\infty}(X,\mu) \rtimes \Gamma$) is called {\it{reflecting}} if every intermediate algebra of the form $C^*_r (\Gamma) < \cA < C(X) \rtimes_r \Gamma$ is of the form $\cA = C(Y) \rtimes_r \Gamma$ for some $\Gamma$-factor $X \rightarrow Y$. The crossed product is called {\it{almost reflecting}} if $\E(\cA) \subset \cA$ for any such intermediate algebra. 
\end{definition}
In the above definition, $\E: C(X) \rtimes_r \Gamma \rightarrow C(X)$ is the canonical conditional expectation. 
\begin{remark}
It is clear that every reflecting crossed product is almost reflecting, but the connection between the two notions goes deeper. In the measurable setting, these two notions are always equivalent. In the topological setting, Suzuki \cite[Proposition~3.4]{Suz17} proves that when $\Gamma$ admits the approximation property (AP), then the two properties are the same. For further details on the connection between these two properties, see Subsection~\ref{subsec:reflecting}. 
\end{remark}

In \cite{A19}, the first named author establishes the reflecting property for non-faithful actions of a large class of $C^*$-simple groups. Suzuki \cite{Suz18} (also see \cite{ryo2021remark}), proves a relative reflecting result to the effect that if $(X,\Gamma) \rightarrow (Y,\Gamma)$ is a factor such that $\Gamma$ acts freely on $Y$; then all the intermediate algebras $C(Y) \rtimes_r \Gamma < \cA < C(X) \rtimes_{r} \Gamma$ satisfy the condition $\E(\cA) \subset \cA$ and hence come from intermediate dynamical factors $X \rightarrow Z \rightarrow Y$ in the case that $\Gamma$ has the property (AP). 

We treat two kinds of $C^*$-simple groups.  $\Gamma$ is called {\it{a convergence group}} if there is a flow $(Z,\Gamma)$ with the property that $\Gamma$ acts properly on $Z^3\setminus \Delta$ where $\Delta$ here is the generalized diagonal. Examples of convergence groups include all Gromov hyperbolic and relatively hyperbolic groups as well as mapping class groups of closed surfaces of negative Euler characteristic and the groups $\Out(F_n)$ and $\Aut(F_n)$. A subgroup $\Gamma < G$ in a locally compact group $G$ is called {\it{a lattice}} if it is both discrete and cofinite in the sense that there is a $G$-invariant probability measure on $G/\Gamma$. With this terminology, our main result is the following: 
\begin{theorem} \label{thm:main}
Assume that $\Gamma$ is either a non-elementary convergence group or a lattice in $G = \SL_d(\R)$ for some $d \ge 2$. Then
\begin{itemize}
\item Every rigid flow $(X,\Gamma)$ is almost reflecting. 
\item Every rigid measurable dynamical system $(X,\cB,\mu,\Gamma)$ is reflecting. 
\end{itemize}
\end{theorem}
\begin{remark}
It is well known that if $\Gamma$ is a Gromov hyperbolic group or a lattice in $\SL_2(\R)$, it admits the approximation property (AP). In these cases, it follows that every rigid $\Gamma$-flow is reflecting.  
\end{remark}

The proofs rely on a tension between the given rigid flow $(X,\Gamma)$ and an auxiliary strongly proximal flow $(Z,\Gamma)$. In the convergence group case, take $(Z,\Gamma)$ to be the minimal non-elementary convergence action. When $\Gamma < \SL_d(\R)$ is a lattice, the role of $Z$ is played by the natural action on the projective space $\bP^{d-1} = \bP^{d-1}(\R)$ or more generally on some Grassman variety $\Gr(l,\R^d)$. 
%In this setting, we define:
In this setup, the following definition will play an essential part:
\begin{definition} \label{def:rsp}
Let $(X,\Gamma)$ be a uniformly rigid $\Gamma$-flow and $(Z,\Gamma)$ a minimal strongly proximal $\Gamma$-flow. A sequence $\{s_i\} \subset \Gamma$ will be called a {\it{rigid strongly proximal sequence}}, or an {\it{RSP sequence}} for short, if it is a rigid sequence on $X$ and at the same time a strongly proximal sequence on $Z$ in the sense that there exists some point $z_0 \in Z$ such that $(s_i)_* \nu \stackrel{w^*}{\longrightarrow} \delta_{z_0}$ for every probability measure $\mu \in \Prob(Z)$.
Similarly, a sequence $\{s_i\}$ is called RSP with respect to a measurable dynamical system $(X,\cB,\mu,\Gamma)$ and $(Z,\Gamma)$ if it is measure rigid on $X$ and strongly proximal on $Z$.
\end{definition}
\begin{theorem} \thlabel{thm:RSP}
Let $\Gamma$ be a countable group and $(Z,\Gamma)$ a $\Gamma$-boundary, i.e. a strongly proximal minimal $\Gamma$-flow.
\begin{itemize}
\item Let $(X,\Gamma)$ be a uniformly rigid flow. Assume that there exists an RSP sequence $\{s_i\} \subset \Gamma$ with respect to the two flows $(X,\Gamma)$ and $(Z,\Gamma)$. Then  $(X,\Gamma)$ is almost reflecting. 
\item Let $(X,\cB,\mu,\Gamma)$ be a rigid measurable dynamical system. Assume that there exists an RSP sequence $\{s_i\}$  with respect to $(X,\cB,\mu,\Gamma)$ and $(Z,\Gamma)$ then $(X,\cB,\mu,\Gamma)$ is reflecting. 
\end{itemize}
\end{theorem}

\subsubsection*{Relation to other papers} As might be clear from the outset, our methods and proofs are strongly influenced by the techniques introduced in Haagerup's beautiful paper \cite{Haagerup}. Theorem \ref{thm:main} generalizes \cite[Corollary~A.2]{A19}, where the first named author, along with Yongle Jiang, showed that $(X,\Gamma)$ is reflecting whenever $\Gamma$ is a residually finite
$C^*$-simple group, and $\Gamma \curvearrowright X$, a free 
exact odometer (profinite) action. The action is called exact if the decreasing sequence of finite index subgroups are normal.
%for probability measure preserving actions $(X,\nu, \Gamma)$,
A similar result has been obtained in \cite{chifan2020rigidity} for compact 
probability measure preserving
ergodic extensions $(Y,\eta, \Gamma)$ of 
$(X,\nu, \Gamma)$, for i.c.c groups $\Gamma$. 
In particular, the authors have shown that the compact ergodic extension $\rho: (Y,\eta, \Gamma)\to (X,\nu, \Gamma)$ is reflecting for an i.c.c. group $\Gamma$ 
%acting probability measure preserving on $\Gamma\curvearrowright (X,\nu)$ 
(see \cite[Theorem~1.3]{chifan2020rigidity}). Moreover, Jiang and Skalski have also independently shown that the profinite action $\Gamma\curvearrowright X$ is reflecting for a residually finite i.c.c. group (see \cite[Corollary~3.11]{jiang2021maximal}).

\subsubsection*{Organization of the paper}
Section \ref{sec:Prelim} is dedicated to terminology and preliminaries. Section \ref{sec:plump} introduces the notion of plump sets and establishes a central technical \thref{mainfaithful}, relating plump sets and the reflecting property. Section \ref{sec:RSP} is dedicated to RSP sequences and the proof of \thref{thm:RSP}. Sections \ref{sec:conv} and \ref{sec:lat} are dedicated, respectively, to the convergence group case and the lattice case of the main Theorem \ref{thm:main}. Finally, in Section \ref{sec:minimalambient}, all these tools are used to construct minimal ambient inclusions with various nice properties. 

\section{Preliminaries} \label{sec:Prelim}

\subsection{\texorpdfstring{$\Gamma$}{Gamma}-flows}
Let $\Gamma$ be a discrete group. 
A {\em $\Gamma$-dynamical system}
(or a $\Gamma$-flow)
$(X, \Gamma)$ (also denoted as $\Gamma \act X$) on a compact (usually metrizable) space $X$, is given by a group homomorphism $\alpha : \Gamma \to \Homeo(X)$, where the latter,
namely, the topological group of homeomorphisms of $X$ is equipped with the topology of uniform convergence.
When $X$ is metrizable, $\Homeo(X)$ is a Polish topological group. Thus the flow is {\it{uniformly rigid}} as defined in the introduction if $\alpha(\Gamma)$ is a non-discrete subgroup of $\Homeo(X)$. We often suppress $\alpha$ from our notation and 
for $x \in X$ and $\gamma \in \Gamma$, write 
$\gamma x$ instead of $\alpha(\gamma) x$.
A dynamical system is {\it minimal} when every orbit $\Gamma x$
is dense in $X$.

An action is called {\em effective} when $\alpha$
is injective. It is {\em free}
when $\gamma x \not =x$ for all $x \in X$ and every 
$e \not =\gamma \in \Gamma$. Finally, it is 
{\em topologically free} when for every
$\gamma \in \Gamma \setminus \{e\}$ the set
$\{x \in X : \gamma x = x\}$ has empty interior.
As $\Gamma$ is always assumed to be countable, it follows that 
$(X, \Gamma)$ is topologically free iff
the set of points $X_0 := \{x \in X : \gamma x \not = x,
\ \forall \gamma \in \Gamma \setminus \{e\}\}$
is a residual subset of $X$. 
In fact, it is easy to see that the subset $X_0$
is a $G_\delta$ subset of $X$; hence, when the action is minimal, topological freeness is equivalent to the above set being nonempty.

\begin{definition} 
The {\em enveloping semigroup} $E(X, \Gamma)$ of a $\Gamma$-flow is defined to be the closure of the image of $\Gamma$ in the product space $X^X$. 
\end{definition}
\begin{definition} 
A dynamical system $(X,\Gamma)$ is called {\em tame} if $E(X, \Gamma)$ is a 
Fr\'{e}chet-Uryshon space, equivalently, if every sequence of elements 
of $E(X,\Gamma)$
admits a convergent subsequence.
\end{definition}

A dynamical system satisfying the equivalent conditions in the following lemma is called {\it{strongly proximal}} 
\begin{lem}\label{lem-SP}
The following properties are equivalent:
\begin{enumerate}
\item 
For every Borel probability measure $\mu$ on $X$, there is a net $\Lambda =\{\gamma_i\} \subset \Gamma$
and a point $x \in X$
such that 
$$
weak^*{\text{-}}\lim (\gamma_i)_*(\mu) \to \delta_x,
$$
\item 
The induced dynamical system $(M(X),\Gamma)$
on the compact space of Borel probability measures
on $X$, endowed with the weak$^*$-topology, is proximal.
\item 
There is a net $\Lambda =\{\gamma_i\} \subset \Gamma$
and a point $x \in X$
such that 
$$
weak^*{\text{-}}\lim (\gamma_i)_*(\mu) \to \delta_x,
$$
for every Borel probability measure $\mu$ on $X$.
\end{enumerate}
If the system $(X,\Gamma)$ is in addition tame,
then $\Lambda$ can be taken to be a {\it sequence}.
\end{lem}
\begin{proof}
(1) $\implies$ (2):
Given $\mu_1, \mu_2 \in M(X)$, form the measure
$\mu = \frac12 (\mu_1 + \mu_2)$. By assumption
there is a net $\{\gamma_i\}_{i \in I} \subset \Gamma$
and a point $x \in X$ 
with $\lim (\gamma_i)_*(\mu) = \delta_x$,
and it follows that $\mu_1$ and $\mu_2$ are proximal.
(2) $\implies$ (3):
The enveloping semigroup of the proximal system $(M(X), \Gamma)$ has a unique minimal left ideal and if
if $u$ is a minimal idempotent in this ideal, then
$u[M(X)] = \{\delta_x\}$, for some point $x \in X$.
Taking $\Lambda =\{\gamma_i\}_{i \in I} \subset \Gamma$
to be a net with $u = \lim \gamma_i$ in $E(M(X), \Gamma)$,
renders the required net.
(3) $\implies$ (1): Clear.
Finally, in a tame system, the Lebesgue-dominated convergence theorem implies that
$E(M(X), \Gamma) = E(X, \Gamma)$, so by definition, $\Lambda$ can be taken to be a sequence.
\end{proof}
A minimal strongly proximal system is often called a $\Gamma$-{\em boundary}.

\subsection{Measurable dynamical systems}
All measurable spaces $(X,\mathcal{B},\mu)$ are considered standard Borel. All actions $\Gamma\curvearrowright (X,\mathcal{B},\mu)$ are probability measure preserving. We always assume that factor maps $\pi: (X,\mu)\to (Y, \eta)$ preserve the measure in the sense that $\pi_*\mu = \eta$. Associated with such a factor map $\pi: (X,\mu)\to (Y, \eta)$, there is the inclusion of the crossed products $L^{\infty}(Y,\eta)\rtimes\Gamma\subset L^{\infty}(X,\mu)\rtimes\Gamma$.  Given a factor $\pi: (X,\mu)\to (Y, \eta)$, there is a bijection between the set of intermediate $\Gamma$-von Neumann subalgebras $L^{\infty}(Y,\eta)\subset\mathcal{N}\subset L^{\infty}(X,\mu)$ and intermediate factors $(X,\mu) \rightarrow (Z,\theta) \rightarrow (Y,\eta)$. 
\subsection{Crossed products}
\label{subsec:crossprod}
Let us briefly recall the construction of the reduced crossed product and refer the readers to \cite{BO:book} for more details.
Let $\mathcal{A}$ be a unital {\em $\Gamma$-$C^*$-algebra}, i.e., the $C^*$-algebra $\mathcal{A}$ is equipped with a group homomorphism $\alpha$ from $\Gamma$ into the group of $*$-automorphisms on $\mathcal{A}$. 
Given a unital $\Gamma$-$C^*$-algebra, 
 let $\pi:\mathcal{A} \to \mathbb{B}(\mathcal{H})$ be a faithful $*$-representation. Denote by $\ell^2(\Gamma,\mathcal{H})$, the space of square summable $\mathcal{H}$-valued functions on $\Gamma$, i.e.,
\[\ell^2(\Gamma,\mathcal{H})=\left\{\xi:\Gamma\to \mathcal{H}\text{ such that }\sum_{t\in\Gamma}\|\xi(t)\|_{\mathcal{H}}^2<\infty.\right\}\]
The action $\Gamma\curvearrowright \ell^2(\Gamma,H)$ is by left translation:
\[\lambda_s\xi(t):=\xi(s^{-1}t), \xi \in \ell^2(\Gamma,\mathcal{H}), s,t \in \Gamma.\]
(Sometimes we write $\lambda(s)$ instead of $\lambda_s$.)
Let $\sigma$ be the $*$-representation \[\sigma:\mathcal{A} \to B(\ell^2(\Gamma,\mathcal{H}))\] defined by \[\sigma(a)(\xi)(t):=\pi(t^{-1}a)\xi(t), a \in \mathcal{A},\]
where $\xi \in \ell^2(\Gamma,\mathcal{H})$ and $t \in \Gamma$.
The {\em reduced crossed product $C^*$-algebra} $\mathcal{A}\rtimes_{\alpha,r}\Gamma$ is the closure in $B(\ell^2(\Gamma,\mathcal{H}))$ of the subalgebra spanned by the operators $\sigma(a)$ and $\lambda_s$, i.e.,
\[\mathcal{A}\rtimes_{\alpha,r}\Gamma=\overline{\text{Span}\left\{\sigma(a)\lambda(s):a\in\mathcal{A},s\in\Gamma\right\}}\]
Note that 
$$\lambda_s\sigma(a)\lambda_{s^{-1}}=\sigma(s.a)$$
for all $s \in \Gamma$ and $a \in \mathcal{A}$. 
The {\em reduced $C^*$-algebra} of $\Gamma$, denoted by $C_r^*(\Gamma)$, is the $C^*$-algebra spanned by $\{\lambda(s): s\in\Gamma\}$.
It follows from the construction that $\mathcal{A}\rtimes_r\Gamma$ contains $C_{\lambda}^*(\Gamma)$ as a $C^*$-sub-algebra. 
The reduced crossed product $\mathcal{A}\rtimes_r\Gamma$ comes equipped with a $\Gamma$-equivariant canonical conditional expectation $\mathbb{E}:\mathcal{A}\rtimes_r\Gamma\to\mathcal{A}$ defined by 
\[
\mathbb{E}\left(\sigma(a_s)\lambda_s\right)=
%\left\{ %\begin{array}{ll}
%0 & \mbox{if $s\ne e$}\\
%\sigma(a_e) & \mbox{otherwise}\end{array}\right\}\]
\begin{cases}
0 & \text{if $s\ne e$}\\
\sigma(a_e) & \text{otherwise}.
\end{cases}
\]
Denote by $\tau_0$ the canonical trace on $C_r^*(\Gamma)$. It is defined by the property that it sends every non-identity element $\lambda_s$ to $0$.

The construction of the von Neumann crossed product is similar to that of the reduced crossed product $C^*$-algebra. Given a faithful $*$-representation $\pi : \mathcal{M}\to \mathbb{B}(\mathcal{H})$ of a $\Gamma$-von Neumann algebra
$\mathcal{M}$ into the space
of bounded operators on the Hilbert space $\mathcal{H}$, the von Neumann crossed product
$\mathcal{M}\rtimes\Gamma$ is generated (as a von Neumann algebra inside $\mathbb{B}(\ell^2
(\Gamma, \mathcal{H})$) by the left regular representation $\lambda$ of $\Gamma$ and a faithful $*$-representation of $\mathcal{M}$ in $\mathbb{B}(\ell^2
(\Gamma, \mathcal{H}))$.
Moreover, this representation translates the action $\Gamma \curvearrowright\mathcal{M}$ into an inner action
by the unitaries $\{\lambda(s), s \in \Gamma\}$. 

When the $\Gamma$-$C^*$-algebra $\mathcal{A}$ is unital and commutative, $\mathcal{A} = C(X)$, where $(X,\Gamma)$ is a $\Gamma$-dynamical system. Thus, with every compact topological dynamical system $(X, \Gamma)$, we associate the reduced  $C^*$-crossed product $C(X) \rtimes_{r} \Gamma$. Similarly with a commutative von Neumann algebras $\mathcal{M} = L^{\infty}(X,\mu)$ one associates the von Neumann algebraic crossed product $L^{\infty}(X,\mu)\rtimes\Gamma$.

\subsection{On reflecting and almost reflecting actions} 
\label{subsec:reflecting}
Let $(X,\Gamma)$ be a compact dynamical system and  
$C^*(\Gamma) \subset \mathcal{A} \subset C(X) \rtimes_r \Gamma$.
Clearly the condition $\mathbb{E}(\mathcal{A}) \subset \mathcal{A}$
is a necessary condition for $\mathcal{A}$ to be of the 
form $C(Y) \rtimes_r \Gamma$, for a factor $X \to Y$.

Suppose this condition is satisfied and
let $\mathcal{K}$ be the algebra generated by 
$\mathbb{E}(\mathcal{A})$.
Then $\mathcal{K} = C(Y)$ for a factor $X \to Y$, and 
$C(Y) \rtimes_r \Gamma \subset \mathcal{A}$.
Given $a \in \mathcal{A}$, by assumption,
the ``Fourier coefficients'' $\mathbb{E}(a \lambda(\gamma))$ are in $C(Y)$
for all $\gamma \in \Gamma$. Now, assuming that $\Gamma$ 
has property-(AP), $a$ can be written as a convergent infinite 
sum $\sum_{\gamma \in \Gamma} c_\gamma E(a\lambda(\gamma))\lambda(\gamma)$, where $c_\gamma
\in \C$ are appropriate weights (see \cite[Theorem~5.4]{crann2022non}), and it follows that
actually $\mathcal{A} = C(Y) \rtimes_r \Gamma$. This establishes the following:

\begin{prop}\cite[Proposition 2.7]{Suz18}\label{remark-AP}
For a group $\Gamma$ with (AP), the condition
$\mathbb{E}(\mathcal{A}) \subset \mathcal{A}$
is necessary and sufficient for an intermediate
algebra $\mathcal{A}$ to be a crossed product of the form
$ C(Y) \rtimes_r \Gamma$ corresponding to a dynamical factor 
$X \to Y$.
Thus, for $\Gamma$ with (AP) being almost reflecting, it is the same as being reflecting.
\end{prop}

In \cite{Suz18} Suzuki constructs
an example (with $\Gamma$ a non-exact group)
of a factor map $(X,\Gamma) \to (Y, \Gamma)$ 
and an intermediate $C(Y) \rtimes_r \Gamma \subset \mathcal{A} \subset C(X) \rtimes_r \Gamma$,
with $\mathbb{E}(\mathcal{A}) \subset C(Y)$,
such that 
$C(Y) \rtimes_r \Gamma \subsetneq \mathcal{A}$. 
\section{Plump sets} \label{sec:plump} 
\begin{definition}
\thlabel{plumpsubset} Let $\Gamma$ be a discrete group. 
Let $\Lambda\subset\Gamma$ be a non-empty subset. 
We say that $\Lambda$ is {\em plump} in $\Gamma$ if the following holds:
given $\epsilon>0$ and finitely many non-identity group elements $t_1,t_2\ldots,t_n\in\Gamma\setminus\{e\}$, we can find $s_1,s_2,\ldots,s_m\in\Lambda$ such that
\[\left\|\frac{1}{m}\sum_{j=1}^m\lambda(s_jt_is_j^{-1})\right\|<\epsilon,~\forall i=1,2,\ldots,n.\]
\end{definition}
This definition is inspired by \cite{Haagerup} and \cite{A19}. It  differs from the notion introduced in \cite{A19} in that $\Lambda$ can be any subset of $\Gamma$ rather than a subgroup. Still, the proof of the following theorem is identical to the one in \cite{Haagerup}. 
\begin{theorem}(See \cite[Theorem~4.5]{Haagerup}.) \thlabel{thm:Haagerup}
Let $\Gamma$ be a group and $\tau_0$ the canonical tracial state on $C^*_r \Gamma$, then the following conditions are equivalent for a subset $\Lambda \subset \Gamma$:
\begin{enumerate}
\item \label{itm:tau_not_cl} $\tau_0 \in \overline{\{s.\phi \ | \ s \in \Lambda \}}^{w^*}$, for any state $\varphi$ on $C^*_r \Gamma$. 
\item \label{itm:tau_not_conv} $\tau_0 \in \overline{\conv}^{w^*} \{s.\phi \ | \ s \in \Lambda \}$, for any state $\varphi$ on $C^*_r \Gamma$. 
\item \label{itm:plump} $\Lambda$ is plump.
\end{enumerate}
Moreover, the existence of a plump subset in $\Gamma$ is equivalent to the $C^*$-simplicity of $\Gamma$. 
\end{theorem}
\begin{definition}
Let $\Gamma$ be a discrete group acting continuously on a compact space $X$ (or, in a measure preserving way, on
a probability measure space $(X,\nu)$). 
\begin{enumerate}
\item For $f\in C(X)$ and $\epsilon>0$, 
%denote the 
set $\Lambda_{\epsilon}^f=\{s\in \Gamma: \|s.f-f\|_{\infty}<\epsilon\}$.
\item For $f\in L^{\infty}(X,\nu)$ and $\epsilon>0$, 
%denote the 
set $\Lambda_{\epsilon}^f=\{s\in \Gamma: \|s.f-f\|_2<\epsilon\}$.
\end{enumerate}
\end{definition}

We can now prove the $C^*$-version of our basic result.
\begin{theorem}
\thlabel{mainfaithful}
Assume for a flow $(X,\Gamma)$ of a countable group that $\Lambda_{\epsilon}^f$ is plump for every $f \in C(X)$ and every $\epsilon>0$. Then $C(X)\rtimes_r\Gamma$ is almost reflecting. 
\begin{proof}
Let $a \in \mathcal{A}$ and $\epsilon>0$. We can find $f_1,f_2,\ldots,f_n\in C(X)$ and $t_1,t_2,\ldots,t_n\in\Gamma\setminus\{e\}$ such that
\begin{equation}
\label{eq:approx}
\left\|a-\left(\sum_{i=1}^nf_i\lambda_{t_i}+\mathbb{E}(a)\right)\right\|<\frac{\epsilon}{3}.\end{equation}
Let $M=\max_{1\le i\le n}\|f_i\|$.
By assumption, $\Lambda_{\frac{\epsilon}{3}}^{\mathbb{E}(a)}$ is plump so there are $s_1,s_2,\ldots,s_m \in \Lambda_{\epsilon}^{\mathbb{E}(a)}$ such that
\[\left\|\frac{1}{m}\sum_{j=1}^m\lambda_{s_j}(\lambda_{t_i})\lambda_{s_j^{-1}}\right\|<\frac{\epsilon}{3nM},~i=1,2,\ldots,n.\]
Using equation~\eqref{eq:approx}, 
\begin{equation}
\label{ineq:C*invarianceineq}
 \left\|\frac{1}{m}\sum_{j=1}^m\lambda_{s_j}\left(a-\left(\sum_{i=1}^nf_i\lambda_{t_i}+\mathbb{E}(a)\right)\right)\lambda_{s_j^{-1}}\right\|<\frac{\epsilon}{3}.   
\end{equation}
Using \cite[Lemma~2.1]{AK}, for each $i=1,2,\ldots,n$, we see that
\begin{equation}
\label{eq:C*boundedabove}
\left\|\frac{1}{m}\sum_{j=1}^m\lambda_{s_j}(f_i\lambda_{t_i})\lambda_{s_j^{-1}}\right\|\le \|f_i\|\left\|\frac{1}{m}\sum_{j=1}^m\lambda_{s_j}(\lambda_{t_i})\lambda_{s_j^{-1}}\right\|
<\frac{\epsilon}{3n}.\end{equation}
Moreover, since $s_j\in \Lambda_{\frac{\epsilon}{3}}^{\mathbb{E}(a)}$ for each $j=1,2,\ldots,m$, 
\begin{equation}
\label{eq:C*functionbound}
\left\|s_j\mathbb{E}(a)-\mathbb{E}(a)\right\|<\frac{\epsilon}{3}.
\end{equation}
Combining equations~\eqref{ineq:C*invarianceineq},\eqref{eq:C*boundedabove} and \eqref{eq:C*functionbound} along with the triangle inequality, let us now observe that
\begin{align*}
&\left\|\frac{1}{m}\sum_{j=1}^m\lambda_{s_j}a\lambda_{s_j^{-1}}-\mathbb{E}(a)\right\|\\
&\le\left\|\frac{1}{m}\sum_{j=1}^m\lambda_{s_j}(a-\mathbb{E}(a))\lambda_{s_j^{-1}}\right\|+ \left\|\frac{1}{m}\sum_{j=1}^m\lambda_{s_j}\mathbb{E}(a)\lambda_{s_j^{-1}}-\mathbb{E}(a)\right\|\\&\le\left\|\frac{1}{m}\sum_{j=1}^m\lambda_{s_j}\left(a-\sum_{i=1}^nf_i\lambda(t_i)-\mathbb{E}(a)\right)\lambda_{s_j^{-1}}\right\|\\&+\sum_{i=1}^n\left\|\frac{1}{m}\sum_{j=1}^m\lambda_{s_j}\left(f_i\lambda(t_i)\right)\lambda_{s_j^{-1}}\right\| +\left\|\frac{1}{m}\sum_{j=1}^m\lambda_{s_j}\mathbb{E}(a)\lambda_{s_j^{-1}}-\mathbb{E}(a)\right\|\\&\le \frac{\epsilon}{3}+\sum_{i=1}^n\|f_i\|\left\|\frac{1}{m}\sum_{j=1}^m\lambda_{s_j}\left(\lambda(t_i)\right)\lambda_{s_j^{-1}}\right\|+\frac{1}{m}\sum_{j=1}^m\|s_j.\mathbb{E}(a)-\mathbb{E}(a)\|\\&\le\frac{\epsilon}{3}+\sum_{i=1}^n\frac{\epsilon}{3n}+\frac{1}{m}\sum_{j=1}^m\frac{\epsilon}{3}=\epsilon.   
\end{align*}
This shows that $\mathbb{E}(a)\in\mathcal{A}$ for all $a \in\mathcal{A}$. In addition, if $\Gamma$ has (AP), the claim is a consequence of \cite[Proposition~3.4]{Suz17}.
\end{proof} 
\end{theorem}

A similar result holds for the von Neumann setup as well.
\begin{theorem}
\thlabel{mainfaithfulvon}
Assume for a measurable dynamical system  $(X,\cB,\mu,\Gamma)$ of a countable group, that $\Lambda_{\epsilon}^f$ is plump for every $f \in L^{\infty}(X,\mu)$ and every $\epsilon>0$. Then $L^{\infty}(X,\nu)\rtimes\Gamma$ is reflecting. 
\end{theorem}
The proof of \thref{mainfaithfulvon} follows the strategy employed in \thref{mainfaithful} except for a few minor changes. We spell out the details nonetheless. For a faithful $\Gamma$-invariant measure, consider the $\Gamma$-invariant state $\varphi$ on $L^{\infty}(X,\nu)\rtimes\Gamma$ defined by  $$\varphi(a)=\nu(\mathbb{E}(a)), a \in L^{\infty}(X,\nu)\rtimes\Gamma.$$ Here $\mathbb{E}$ is the canonical conditional expectation from $L^{\infty}(X,\nu)\rtimes\Gamma$ onto $L^{\infty}(X,\nu)$. Then $\varphi$ is a faithful normal state on $L^{\infty}(X,\nu)\rtimes\Gamma$ which is $\Gamma$-invariant. Consider the $\|.\|_2$-norm on $L^{\infty}(X,\nu)\rtimes\Gamma$ associated with $\varphi$, defined by \[\|a\|_2:=\sqrt{\varphi(a^*a)} \text{ ~ for } a \in L^{\infty}(X,\nu)\rtimes\Gamma.\] 

\noindent
Let $a\in\mathcal{M}$. By definition, for any $s\in\Gamma$,
\[\left\|\lambda(s)a\lambda(s)^*\right\|_2^2=\nu\left(\mathbb{E}\left(\left(\lambda(s)a\lambda(s)^*\right)^*\left(\lambda(s)a\lambda(s)^*\right)\right)\right)
=\nu\left(\mathbb{E}\left(\lambda(s)a^*a\lambda(s)^*\right)\right)\]
Since $\mathbb{E}$ is $\Gamma$-equivariant, 
\[\mathbb{E}\left(\lambda(s)a^*a\lambda(s)^*\right)=\lambda(s)\mathbb{E}(a^*a)\lambda(s)^*=s.\mathbb{E}(a^*a).\] Therefore,
$$\left\|\lambda(s)a\lambda(s)^*\right\|_2^2=\nu\left(\mathbb{E}\left(\lambda(s)a^*a\lambda(s)^*\right)\right)=\nu\left(s.\mathbb{E}(a^*a)\right)
=\nu\left(\mathbb{E}(a^*a)\right)=\|a\|_2^2.
$$
The last equality follows since $\nu$ is $\Gamma$-invariant. 
So that
\begin{equation}
\label{eq:invariantnormin}
\left\|\lambda(s)a\lambda(s)^*\right\|_2
= \|a\|_2,~\forall a\in\mathcal{M},~\forall s\in\Gamma\end{equation}

\begin{proof}[Proof of \thref{mainfaithfulvon}]
Let $a \in \mathcal{M}$ and let $\epsilon > 0$ be given. 
We can find $t_1,t_2,\ldots,t_n \in \Gamma\setminus\{e\}$ such that 
\[\left\|a-\left(\sum_{i=1}^nf_i\lambda_{t_i}+a_e\right)\right\|_2 < \frac{\epsilon}{6}.\]
Since $\varphi$ is $\mathbb{E}$-invariant, $\mathbb{E}$ is continuous with respect to the $\|.\|_2$-norm and hence by the triangle inequality,
\begin{equation}
\label{eq:primeineq}
 \left\|a-\left(\sum_{i=1}^nf_i\lambda_{t_i}+\mathbb{E}(a)\right)\right\|_2 <\frac{\epsilon}{3}.   
\end{equation}
Let $M=\max_{1\le i\le n}\|f_i\|$.
By assumption, 
$\Lambda_{\frac{\epsilon}{3}}^{\mathbb{E}(a)}$ is plump and hence, we can find $s_1,s_2,\ldots,s_m \in \Lambda_{\epsilon}^{\mathbb{E}(a)}$ such that
\[\left\|\frac{1}{m}\sum_{j=1}^m\lambda_{s_j}(\lambda_{t_i})\lambda_{s_j^{-1}}\right\|<\frac{\epsilon}{3nM},~i=1,2,\ldots,n.\]
In particular, using the triangle inequality coupled with equation~\eqref{eq:invariantnormin} above, it follows from equation~\eqref{eq:primeineq} that
\begin{equation}
\label{ineq:invarianceineq}
 \left\|\frac{1}{m}\sum_{j=1}^m\lambda_{s_j}\left(a-\left(\sum_{i=1}^nf_i\lambda_{t_i}+\mathbb{E}(a)\right)\right)\lambda_{s_j^{-1}}\right\|_2<\frac{\epsilon}{3}.   
\end{equation}
Using \cite[Lemma~2.1]{AK} along with the fact that the $\|.\|_2$ is dominated by the operator norm, yields for each $i=1,2,\ldots,n$,
\begin{equation}
\label{eq:boundedabove}
\left\|\frac{1}{m}\sum_{j=1}^m\lambda_{s_j}(f_i\lambda_{t_i})\lambda_{s_j^{-1}}\right\|_2\le \|f_i\|\left\|\frac{1}{m}\sum_{j=1}^m\lambda_{s_j}(\lambda_{t_i})\lambda_{s_j^{-1}}\right\|
<\frac{\epsilon}{3n}.\end{equation}
Moreover, since $s_j\in \Lambda_{\frac{\epsilon}{3}}^{\mathbb{E}(a)}$ for each $j=1,2,\ldots,m$, 
\begin{equation}
\label{eq:functionbound}
\left\|s_j\mathbb{E}(a)-\mathbb{E}(a)\right\|_2<\frac{\epsilon}{3}.
\end{equation}
Combining equations~\eqref{ineq:invarianceineq},\eqref{eq:boundedabove} and \eqref{eq:functionbound} along with the triangle inequality, let us now observe that
\begin{align*}
&\left\|\frac{1}{m}\sum_{j=1}^m\lambda_{s_j}a\lambda_{s_j^{-1}}-\mathbb{E}(a)\right\|_2\\&\le\left\|\frac{1}{m}\sum_{j=1}^m\lambda_{s_j}(a-\mathbb{E}(a))\lambda_{s_j^{-1}}\right\|_2+ \left\|\frac{1}{m}\sum_{j=1}^m\lambda_{s_j}\mathbb{E}(a)\lambda_{s_j^{-1}}-\mathbb{E}(a)\right\|_2\\&\le\left\|\frac{1}{m}\sum_{j=1}^m\lambda_{s_j}\left(a-\sum_{i=1}^nf_i\lambda(t_i)-\mathbb{E}(a)\right)\lambda_{s_j^{-1}}\right\|_2\\&+\sum_{i=1}^n\left\|\frac{1}{m}\sum_{j=1}^m\lambda_{s_j}\left(f_i\lambda(t_i)\right)\lambda_{s_j^{-1}}\right\|_2 +\left\|\frac{1}{m}\sum_{j=1}^m\lambda_{s_j}\mathbb{E}(a)\lambda_{s_j^{-1}}-\mathbb{E}(a)\right\|_2\\&\le \frac{\epsilon}{3}+\sum_{i=1}^n\|f_i\|\left\|\frac{1}{m}\sum_{j=1}^m\lambda_{s_j}\left(\lambda(t_i)\right)\lambda_{s_j^{-1}}\right\|+\frac{1}{m}\sum_{j=1}^m\|s_j.\mathbb{E}(a)-\mathbb{E}(a)\|_2\\&\le\frac{\epsilon}{3}+\sum_{i=1}^n\frac{\epsilon}{3n}+\frac{1}{m}\sum_{j=1}^m\frac{\epsilon}{3}=\epsilon.   
\end{align*}
Since $\epsilon>0$ is arbitrary, this shows that
$\mathbb{E}(a) \in \mathcal{M}$. The proof is complete by \cite[Corollary 3.4]{Suz18}. 
\end{proof}

\section{Rigid strongly proximal sequences.} \label{sec:RSP}
This section is dedicated to the proof of \thref{thm:RSP}.
\begin{proof}
We prove the $C^*$ version. The von Neumann proof is almost identical and left to the reader. Given an RSP sequence $\{s_i\} \subset \Gamma$, which is uniformly rigid for the flow $(X,\Gamma)$ and strongly proximal for $(Z,\Gamma)$, we will show that $(X,\Gamma)$ is almost reflecting.

Fix $f \in C(X), \  \epsilon_1 > 0$. Since $f$ is uniformly continuous we have $\Gamma_{\epsilon} := \Gamma \cap B_{\Homeo(X)}(\epsilon) \subset \Lambda^{f}_{\epsilon_1}$ for an appropriately chosen $\epsilon$. By \thref{mainfaithful}, it would be enough to prove that $\Gamma_{\epsilon}$ is plump for every $\epsilon > 0$.

By \thref{thm:Haagerup} (\cite[Theorem~4.5]{Haagerup}), it suffices to establish that 
\[\tau_0 \in \overline{\text{Conv}\left\{\Gamma_{\epsilon} \varphi\right\}}^{\text{weak}^*}, \qquad {\text{for every state }} ~\varphi\in S(C_r^*(\Gamma)).\] 
For this, fix a state $\varphi$ on $C_r^*(\Gamma)$ and extend it to a state (which by abuse of notation is still denoted $\varphi$) on $C(Z)\rtimes_r\Gamma$. 
By assumption have, for each $g \in \Gamma$, 
$$\lim_{n \rightarrow \infty} (gs_ng^{-1})_* \varphi|_{C(Z)}=\delta_{ga},$$ 
where convergence of measures is always taken in the weak$^{*}$ topology. By the compactness of the set of states in this topology, and upon passing to a subsequence if required, we may assume that for every $g \in \Gamma$
$$\lim_{n \rightarrow \infty} gs_ng^{-1} \varphi = \psi_g \qquad {\text{for some state }} \psi_g \in S(C(Z)\rtimes_r\Gamma).$$
Since $gs_ng^{-1}$ converges uniformly to $1 \in \Homeo(X)$,
\begin{equation} \label{eqn:psi_g}
\psi_g \in \overline{\text{Conv}\left\{\Gamma_{\epsilon} \varphi\right\}}^{\text{weak}^*}, \qquad \forall g \in \Gamma
\end{equation}
We claim that $\psi_g|_{C_r^*(\Gamma)}=\psi_g \circ \mathbb{E}_{ga}$; where $\mathbb{E}_{ga}: C_r^*(\Gamma) \to C_r^*(\Gamma_{ga})$ is the canonical conditional expectation on the reduced $C^*$-algebra of the stabilizer of the point $ga$. To show this, it suffices to show that $\psi_g(\lambda(t)) = 0$ for $t \not\in\Gamma_{ga}$. 
Since $\psi_g|_{C(Z)}=\delta_{ga}$, $C(Z)$ lies in the multiplicative domain of 
$\psi_g$. Let $t\not\in\Gamma_{ga}$. Since $tga\ne ga$, 
one can find an element $f \in C(Z)$ such that $f(ga) = 1$ and $f(tga) = 0$. 
Now:
    \begin{align} \label{eqn:stab_vanish}
\nonumber  \psi_g(\lambda(t))
\nonumber    &=f(ga) \psi_g(\lambda(t))\\
\nonumber    & = \psi_g(f \lambda(t))\\
\nonumber    & 
    = \psi_g(\lambda(t) (t^{-1} \cdot f))\\
    & = \psi_g(\lambda(t) f(tga) =0.
    \end{align}
     
Let us now enumerate $\Gamma\setminus\{e\}=
\{\gamma_1,\gamma_2,\ldots,\gamma_n,\ldots\}$. 
Since $\Gamma\curvearrowright Z$ is minimal and topologically free, for each $n\in\mathbb{N}$, we can find some $g_n\in\Gamma\setminus\{e\}$ such that $$g_na\in Z\setminus\left(\bigcup_{i=1}^n\text{Fix}(\gamma_i)\right).$$
Thus by Equation (\ref{eqn:stab_vanish})
\begin{equation} \label{eqn:psi_gn}
\psi_{g_n}(\lambda(\gamma_k)) = 0, \ \forall n \ge k.
\end{equation}
Let $\theta_n=\frac{1}{n}\sum_{i=1}^n\psi_{g_i}$. Using Equation (\ref{eqn:psi_gn}) for the $\stackrel{\star}{=}$ below yields,  
\begin{equation} \label{eqn:allvanish}
\theta_n(\lambda(\gamma_k)) = \frac{1}{n} \sum_{i=1}^{n} \psi_{g_i}(\lambda(\gamma_k)) \stackrel{\star}{=} \frac{1}{n} \sum_{i=1}^{k} \psi_{g_i}(\lambda(\gamma_k)) \le \frac{k-1}{n}, \quad \forall n \ge k.
\end{equation}
Assume, after passing to a subsequence, that this sequence of linear functionals converges $\tau = \lim_i\theta_{n_i}$. Equation ( \ref{eqn:allvanish}) implies that  $\tau(\lambda(\gamma_k)) = 0, \ \forall k$. Since $\{\gamma_k\}$ are an enumeration of all nontrivial elements in the group, we conclude that $\tau = \tau_0$. As each $\theta_n \in \overline{\text{Conv}\left\{\Gamma_{\epsilon} \varphi \right\}}^{\text{weak}^*}$, the same holds for $\tau_0 = \tau$, which finishes the proof of the $C^*$-case.  
\end{proof} 

\begin{remark} \label{rem:minimal_top_free}
It follows from the work of Kalantar and Kennedy \cite{KK} that $\Gamma$ is $C^*$-simple if and only if it admits a free boundary action. Our proof above is inspired by Haagerup's proof of Theorem \ref{thm:Haagerup}. More specifically, by the implication that $C^*$-simplicity (i.e., the existence of a free boundary) implies Item (\ref{itm:tau_not_conv}) of that theorem. However, in our proof, we make do with a topologically free boundary rather than a free boundary. Thus, in particular, we prove that the existence of a topologically free boundary is already equivalent to the $C^*$-simplicity of the group. 

This is not new. It was already observed in \cite{KK} that the existence of a topologically free boundary guarantees the freeness of the action of the group on its universal (Furstenberg) boundary. Still, our proof above gives a new, more direct proof of this fact.
\end{remark}

 \section{Convergence groups}
 \label{sec:conv}
We briefly recall the notion of convergence action. Given
an action $\Gamma \curvearrowright Z$, it is said to be a convergence action, if for every infinite sequence of distinct elements $\gamma _{n}\in \Gamma$, there exists a subsequence $\gamma _{n_{k}}$ and points $a,b\in Z$ such that $\gamma _{n_{k}}|_{Z\setminus \{a\}}$ converge uniformly on compact subsets to $b$. If $\Gamma$ admits such an action on a set with at least three points, without a global fixed point,  then $\Gamma$ is called a non-elementary convergence group. A non-torsion element $s \in \Gamma$ is called loxodromic if $s$ has exactly two fixed points and is called parabolic if $s$ fixes exactly one point. We refer the reader to \cite{Bow, Fred, Tuk93} for more details.

Let us also briefly recall the notion of measure-theoretic rigidity from \cite{furstenberg1978finite}. It is easy to verify that this is equivalent to the definition given in the introduction. 
\begin{definition}
A measurable probability preserving system $(X,\mu, \Gamma)$ is called {\em rigid} if there exists a sequence $\{\gamma_n:n\in\mathbb{N}\}=\Lambda\subset\Gamma$ such that $\lim_{n \rightarrow \infty} \norm{\gamma_nf-f}_2 = 0$ for all $f\in L^2(X,\mu)$. 
A function $f\in L^2(X,\mu)$ is called {\em rigid} if there 
is a sequence $\{\gamma_n:n\in\mathbb {N}\}=\Lambda\subset\Gamma$ (depending on $f$) such that $\lim_{n \rightarrow \infty} \norm{\gamma_n f -f}_2 =0$ .\end{definition}

It follows from \cite[Corollary~2.6]{bergelson2014rigidity} that $(X,\mu, \Gamma)$ is rigid if and only if every $f\in L^2(X,\mu)$ is rigid.

\begin{theorem}\thlabel{con-RSP}
Let $\Gamma$ be a non-elementary convergence group with no finite normal subgroups. 
\begin{enumerate}
    \item Every uniformly rigid flow $(X, \Gamma)$ is almost reflecting and actually reflecting if $\Gamma$ is assumed to have 
    %property-
    (AP).
    \item Every p.m.p rigid $\Gamma$-space $(X,\mu,\Gamma)$ is reflecting.
\end{enumerate}
\end{theorem}
\begin{proof}
Let $\Gamma\curvearrowright Z$ be a non-elementary convergence action. Replacing $Z$ with the limit set, we may assume that the action is minimal. Since every $e \ne \gamma \in \Gamma$  has at most two fixed points, the action is topologically free. It will be shown below that the action is strongly proximal. It  is enough to exhibit an RSP sequence for $(Z,\Gamma)$ and $(X,\Gamma)$
(or, $(X,\mu, \Gamma)$) and then conclude the proof using \thref{thm:RSP}. 

Let $1 \ne \gamma_n \in \Gamma$ be any uniformly rigid sequence on $X$. The  convergence axiom yields a subsequence $\{n_i\} \subset \Z$ and $a,r \in Z$ such that $\lim_{i \rightarrow \infty}\gamma_{n_i} z = a$ for every $z \in Z \setminus \{r\}$ and the convergence is uniform on compact subsets of $Z \setminus \{r\}$. Passing to a further subsequence (which by abuse of notation is still denoted by $\gamma_{n_i}$), we can assume in addition that $\lim_{i \rightarrow \infty} \gamma_{n_i} r = v$ for some $v \in Z$. By the minimality of the action, there is an element $g \in \Gamma$ such that $g^{-1}r \not \in \{a,v\}$. We claim that the sequence $\{s_i = \gamma_{n_i} g \gamma_{n_i} g^{-1}\}$ satisfies all the conditions of an RSP sequence.  Indeed, $g \gamma_n g^{-1}$ is rigid as a conjugate of a rigid sequence so that $\gamma'_n$ is rigid as a product of two rigid sequences of homeomorphisms of $X$. On $Z$ we claim that
$(s_i)_{*} \mu \xrightarrow{\text{weak}^*} \delta_{ a}, \forall\mu\in\text{Prob}(Z)$, for some $a \in Z$. 
This is equivalent to showing that for every $\mu\in\text{Prob}(Z)$,
\begin{equation}
\label{eq:equivconv}
\int_Z f(s_i z)d\mu(z) = \int_Z f(\gamma_{n_i} g \gamma_{n_i} g^{-1} z)d\mu(z)\xrightarrow[]{}f(a),~\forall f\in C(Z).    
\end{equation} 
For $f \in C(Z)$ denote $f_i=f \circ s_i$. It would be enough to show that $f_i(z)\xrightarrow[]{i\to\infty}f(a)$ pointwise for every $z\in Z$. Since $\|f_i\|_{\infty}= \|f\|_{\infty} < \infty$, Equation~\eqref{eq:equivconv} would then follow by Lebesgue's dominated convergence theorem. 

Let $f\in C(Z)$, $\epsilon>0$ and $z \in Z$ be given. By continuity there exists a $\delta>0$ such that $|f(a)-f(y)|<\epsilon$ for all $y\in B_{\delta}(a)$. Since by construction $ga,gv \ne r$ there is some $\eta > 0$ and $I$ such that $\gamma_{n_i} \left(B_{\eta}(ga) \cup B_{\eta}(gv)\right) \subset B_{\delta}(a)$ for every $i \ge I$. 

Let us distinguish between two cases. If $g^{-1}z \ne r$ then for some $I_1$ we have $g \gamma_{n_i} g^{-1} z \in B_{\eta}(ga)$ for any $i \ge I_1$ and consequently $s_i z = \gamma_{n_i} g \gamma_{n_i} g^{-1}z \in B_{\delta}(a)$ and by the choice of $\delta$,  $\abs{f(s_i z)- f(a)} < \epsilon$ whenever $i \ge \max\{I,I_1\}$, as desired. 
Similarly if $g^{-1}z = r$, then $g \gamma_{n_i} g^{-1} z \in B_{\eta}(gv)$ for any $i \ge I_2$ and we conclude in the same way. 
 \end{proof}
\begin{remark}
In particular, the above theorem holds for any 
equicontinuous $\Gamma$-space $(X,d)$, where $\Gamma$ is a convergence group. 
\end{remark}
\begin{remark}
If $\Gamma$ is a Gromov hyperbolic group, then it is automatically a convergence group and has (AP) (see for example~\cite{Oz08}). 
\end{remark}
\begin{remark}
Examples of rigid p.m.p. $\Gamma$-systems
are prevalent. For example, for a residually finite convergence group $\Gamma$ (see Section~\ref{sec:conv}), the action of $\Gamma$ on its profinite completion is equicontinuous, hence rigid, and preserves the
Haar measure on this compact group.

For examples of such actions of the free group $\Gamma = F_2$, start with any rigid 
transformation $(X,\mu, T)$. Such transformations, which are weakly mixing,
are abundant in the literature on ergodic theory.
Now, for any other $\mu$-preserving
transformation $S : (X, \mu) \to (X,\mu)$,
the acting group $\Gamma =\langle T, S \rangle$,
can be considered as a p.m.p action of $F_2$
which is rigid.

Note that when $(X,\Gamma)$ is a metrizable $\Gamma$-space which 
is uniformly rigid, then by the Lebesgue-dominated convergence theorem, it follows that $(X,\mu, \Gamma)$ is rigid (in the measurable sense) for any invariant probability measure $\mu$ on $X$. A weaker notion of topological rigidity
(called {\em rigid} in \cite{GlasnerMaon}), where the rigidity sequence 
converges only pointwise to the identity map on $X$, will suffice.
\end{remark}

\section{Lattices in \texorpdfstring{$\SL_d(\R)$}{SL(d,R)}} \label{sec:lat}
Let $\Gamma < G = \SL_d(\R), \ d \geq 2$, be a lattice. Denote by $\bP^{d-1}$ the projective space of all lines in $\R^d$. For a non-zero vector $v \in \R^d$ we let $\overline{v}$ denote its image in $\mathbb{P}^{d-1}$. For a subspace $W < \R^d$, denote the corresponding projective subspace by $\overline{W} = \{\overline{w} \ | \ w \in W\} < \bP^{d-1}$. For $g \in G$, let $\overline{g}$ denote its image in the projective group $\PSL(d,\R)$.

Instead of using the convergence axiom, we appeal to Furstenberg's  lemma \cite{furstenberg1976note}. For the reader's convenience, let us review the proof of Furstenberg's lemma in terms of tame systems and enveloping semigroups. 
 \begin{lem}[Furstenberg's lemma \cite{furstenberg1976note}]\label{Flem} 
The action of $G = \SL_d(\R)$ on $\mathbb{P}^{d-1}$ is tame,
and for every element $p \in E(\mathbb{P}^{d-1},\SL_d(\R)) \setminus \SL_d(\R)$, the image
$p [\mathbb{P}^{d-1}]$ is a finite union of projective 
sub-varieties of dimension $< d-1$, 
({\em a quasi-sub-variety}).
%and strongly proximal.
\end{lem}
\begin{proof}
%(Sketch)
Let $g_n$ be a matrix sequence in $\GL_d(\R)$. Given any subspace $W \subset \R^d$, by
passing to a subsequence and by choosing appropriate scalars $\lambda_n$, we can assume that
$\lambda_n g_n|_{W} \to h_W$, a non-zero linear map from $W$ into $\R^d$. For $v \not\in \ker(h_W)$ we have $\overline{\lambda_ng_n(v)} \to \overline{h_W(v)}$.

Now proceed by induction. Define $W_0 =\R^d$ and using the above procedure obtain a linear map $h_{W_0}$ with $W_1 := \ker(h_{W_0})$ and $L_1 := \Im(h_{W_0})$, such that after replacing $\{g_n\}$ by a subsequence  $\lim_{n \rightarrow \infty} \overline{g}_n \overline{x} \rightarrow \overline{h_{W_0} x}$, for every $x \in W_0 \setminus W_1$. Now repeat this procedure with $W_1$ to obtain, after passing to a further subsequence, that $\lim_{n \rightarrow \infty} \overline{g}_n \overline{x} \rightarrow \overline{h_{W_1} x}, \ \forall x \in W_1 \setminus W_2$, for some map $h_{W_1}: W_1 \rightarrow L_2$, with $\ker(h_{W_1}) := W_2$. Proceeding by induction, until reaching a trivial kernel, we finally obtain a subsequence such that $\overline{g_{n_i} }\to p \in E(\mathbb{P}^{d-1}, G)$.
This shows that $(\mathbb{P}^{d-1}, G)$ is tame.

Note that for $W_0 =\R^d$ the assumptions $g_n \in \SL_d(\R)$ 
and $\|g_n\| \to \infty$, imply that $\dim \ker h_W \ge 1$. 
Thus, in this case, the image of $p$ is a proper quasi-sub-variety.
\end{proof}
\begin{remark} \label{rem:dim_furst}
As remarked at the end of the proof, if $\norm{g_n} \rightarrow \infty$ then $p(\bP^{d-1} \setminus \overline{W}_1) \subset \overline{L}_1$ with $\dim(\overline{W}_1) + \dim(\overline{L}_1) = d-2$. Where $\dim$ denotes the projective dimension of the space. It is quite possible though to have $p(\bP^{d-1} \setminus \overline{W}) \subset \overline{L}$ with $\dim(\overline{L}) + \dim(\overline{W})$ much smaller. For example, in the second step of the proof, it might very well happen that $L_2 = h_{W_1}(W_1) < L_1$, in which case already $p(\bP^{d-1} \setminus \overline{W}_2) \subset \overline{L}_1$. This will be important in the proof below. There, we allow, slightly more generally, for $W$ to be a finite union of linear subspaces. 
\end{remark}

\begin{theorem}\label{lattice-proximal} 
Let $\Gamma$ be a lattice in $\SL_d(\R)$. Then, every uniformly rigid $\Gamma$-flow $(X,\Gamma)$ is almost reflecting.
\end{theorem}
\begin{proof}
Consider the action of $\Gamma$ on $\bP^{d-1}(\bR)$ by projective linear transformations. We treat the situation where a sequence of elements $\{\gamma_n \ | \ n \in \N\}$ is a rigid sequence on $X$ and at the same time convergent in the Ellis group $\lim_{n \rightarrow \infty} \overline{\gamma}_{n} = p \in E(\bP^{d-1},\Gamma)$ on the projective space $\bP^{d-1}$ satisfying
\begin{equation} \label{eqn:best_dim}
p(\bP^{d-1} \setminus \overline{W}) \subset \overline{L}.
\end{equation}
Here $L < \R^n$ is a subspace, and $W \subset \R^n$ is a finite union of linear subspaces. Let us further assume that  the dimensions of these projective spaces: 
\begin{eqnarray*}
c & = & \dim(\overline{W}) := \max \{\dim(\overline{w}) \ | \ w \subset W {\text{ is a linear subspace}}\} \\
l & = & \dim(\overline{L})
\end{eqnarray*}
are chosen to be the minimal possible. Namely, $l$ is chosen to be minimal, ranging over {\bf{all}} possible rigid sequences $\gamma_n \in \Gamma$ as above, and then $c$ is chosen to be the minimal possible among all sequences that realize this value of $l$. 

Furstenberg's lemma gives such a sequence and an a priori upper bound on the dimensions. Indeed, apply this lemma to any rigid sequence $\{\gamma_n\} \subset \Gamma$. Since $\Gamma$ is discrete, and the sequence is by assumption not eventually constant, $c+l \le d-2$ by Remark \ref{rem:dim_furst}.

\begin{claim}
$c = -1$, that is $W = \emptyset$.
\end{claim}
\begin{proof}
Assume by way of contradiction that $c > -1$. Let us find some $s \in \Gamma$ such that 
\begin{enumerate}
\item \label{itm:i} $\overline{L} \cap  s \overline{W} = \emptyset$ and $\overline{L} \cap s^{-1} \overline{W} = \emptyset$
\item \label{itm:ii} $p(\overline{w}) \not \subset s \overline{W}$ for every one of the finite dimensional subspaces $w$ comprising $W$.
\end{enumerate}
Since $\dim(\overline{L}) + dim(\overline{W}) \le d-2$ the sets $\Omega^{1a} = \{s \ | \ \overline{L} \cap s\overline{W} = \emptyset\}$ and $\Omega^{1b} = \{s \ | \ \overline{L} \cap s^{-1}\overline{W} = \emptyset \}$ are both Zariski open and nonempty in $\SL_n(\C)$. Similarly, since $c < d-1$, for any $w \subset W$ as in (\ref{itm:ii}) the set $\Omega^2_{w} = \{s \in \SL_d(\C) \ | \ p(w) \not \subset s \overline{W} \}$ is also a nonempty Zariski open subset. Consequently, the finite intersection 
$$\Omega = \Omega^{1a} \cap \Omega^{1b}\cap \left(\bigcap_{w \subset W_k} \Omega^2_{w}\right),$$ 
nonempty, and Zariski open. By Borel's density theorem, $\Gamma$ is Zariski dense, so we can find some $s \in \Gamma \cap \Omega \ne \emptyset$, which we wanted.

Using the element $s$, define a new sequence $\gamma'_n = s \gamma_n s^{-1} \gamma_n$. Note that $s \gamma_n s^{-1}$ is rigid as a conjugate of a rigid sequence, and hence $\gamma'_n$ is rigid as a product of two rigid sequences of homeomorphisms of $X$. Since the system $(\bP^{d-1},\Gamma)$ is tame, we can assume, after replacing it with a subsequence, that it is also convergent in the Ellis group. Thus write $p' = \lim_{n \rightarrow \infty} \gamma'_n \in E(\bP^{d-1},\Gamma)$. 

Now, let us define 
$$W' = \left\{z \in W \ | \ p(z) \in sW \right\}, \qquad L' = s L.$$
Condition (\ref{itm:ii}) above guarantees that
$$c' = \max \left\{\dim(\overline{w}') \ | \ \overline{w}' {\text{ projective subspace }}, \overline{w}' \subset \overline{W}' \right\}  \lneqq c.$$

Finally we claim that $p'(\bP^{d-1} \setminus  \overline{W}') \subset \overline{L}'$. For this, let us fix any metric on $\bP^{d-1}$ compatible with the topology and denote by $(C)_{\epsilon}$ the $\epsilon$-neighbourhood of a set $C$ with respect to this metric. Now fix some $\overline{y} \in \bP^{d-1} \setminus \overline{W}'$ and let $\epsilon > 0$ be given. By continuity of $s$ and the compactness projective spaces, there is some $\epsilon_1 > 0$ such that 
\begin{equation} \label{step:3} 
s (\overline{L})_{\epsilon_1} \subset (\overline{L}')_{\epsilon}.
\end{equation} 
By the definition of $W'$, we 
know that $s^{-1} p(\overline{y}) \not \in \overline{W}$. Let $\epsilon_2$ be small enough so that the closed ball  $B = \overline{\left(s^{-1}p(\overline{y}) \right)_{\epsilon_2}}$ is still disjoint from $\overline{W}$. Thus by the construction of $\{\gamma_n\}$ we have $\lim_{n \rightarrow \infty} \gamma_n b \in \overline{L}, \ \forall b \in B$. By the compactness of $B$ there is some $N_1 \in \N$ such that 
\begin{equation} \label{step:2}
\gamma_n (B) \subset (\overline{L})_{\epsilon_1}, \qquad \forall n \ge N_1
\end{equation}
Finally since $\lim_{n \rightarrow \infty} \gamma_n \overline{y} = p(\overline{y})$ there is some $N_2 \in \N$ such that 
\begin{equation} \label{step:1}
s^{-1}\gamma_n \overline{y} \in B, \qquad \forall n \ge N_2
\end{equation}
Combining Equations (\ref{step:1}),(\ref{step:2}) and (\ref{step:3}) we conclude that for any $n \ge \max\{N_1,N_2\}$,
$$\gamma'_n \overline{y} = s \gamma_n s^{-1} \gamma_n \overline{y} \stackrel{(\ref{step:1})}{\in} s \gamma_n (B) \stackrel{(\ref{step:2})}{\subset} s ((L)_{\epsilon_1}) \stackrel{(\ref{step:3})}{\subset} (L')_{\epsilon}$$
Since $\epsilon$ and $\overline{y}$ were arbitrary it follows that $
\lim_{n \rightarrow \infty} \gamma'_n \overline{y} \in L'$ for every $\overline{y} \in \bP^{d-1} \setminus \overline{W}'$.
Our new sequence $\gamma_n'$ satisfies $
p(\bP^{d-1} \setminus \overline{W}) \subset \overline{L}$, contrary to the assumed minimality of $\dim(\overline{W}))$ in Equation \ref{eqn:best_dim}. This finishes the proof of the claim
\end{proof}
Now let $Z = \Gr(l,\R^d)$ be the Grassmann variety of $l$-dimensional subspaces of $\R^d$. The $l$-dimensional projective subspace $L \in Z$ can be considered a point in this space and let $z \in Z$ be any other point. We now know that $p(\overline{y}) = \lim_{n\rightarrow \infty} \gamma^n \overline{y} \in L$ for every $\overline{y} \in \bP^{d-1}$. By the compactness of projective spaces, we deduce that for every $\epsilon > 0$, there exists an $N \in \N$ such that $\gamma^n(z) \in (L)_{\epsilon}, \ \forall n \ge N$. In other words, in the Grassmannian $\lim_{n \rightarrow \infty}z = L$. Now Lebesgue's dominated convergence theorem yields a $w^*$ convergence $\lim_{n \rightarrow \infty} (\gamma_n)_* \mu = \delta_{L}$ for any probability measure $\mu$ on $Z$. Thus, the sequence $\{\gamma_n\}$ is an RSP sequence for the pair of dynamical systems $(X,\Gamma)$ and $(Z,\Gamma)$, and we conclude by appealing to Theorem \ref{thm:RSP}. 
\end{proof}

\section{Minimal ambient inclusions}
\label{sec:minimalambient}

\begin{definition}
\thlabel{def:minimalambient}   
A proper inclusion $\mathcal{B}\subset\mathcal{A}$ of unital $C^*$-algebras is called {\em minimal ambient} if 
there is no intermediate  $C^*$-algebra $\mathcal{C}$ with 
$\mathcal{B}\subsetneq\mathcal{C}\subsetneq\mathcal{A}$. We define a similar notion, with the same name, for a proper inclusion of von Neumann algebras $\mathcal{N}\subset\mathcal{M}$.
\end{definition} 

\begin{definition}
\thlabel{def:prime}   
A dynamical system, either topological or measure-preserving, is {\em prime} if it admits no nontrivial factor.
\end{definition} 

The following proposition is then clear.

\begin{prop}\label{am}
If $(X,\Gamma)$ is a dynamical system that is both reflecting and prime, then the inclusion
$\C^*_r(\Gamma) \subset C(X) \rtimes_r \Gamma$ is minimal ambient.
An analogous statement holds for probability measure preserving actions and their corresponding von Neumann algebras.
\end{prop} 
The following proposition is a rich source for prime, rigid actions. When the group $\Gamma$ is Gromov hyperbolic (e.g., a nonabelian free group), the combination with Theorem \ref{thm:main} yields examples of minimal ambient inclusions. 
\begin{theorem} \label{thm:minimal_ambient}
Let $G$ be a topological group acting transitively on a compact space $X = G/L$, where $L < G$
is a maximal closed subgroup. Let $\Gamma< G$ be a dense countable subgroup then:
\begin{enumerate}
\item
Both systems $(X, G)$  and $(X,\Gamma)$ are prime.
\item
Let $\mu$ be a fully supported probability measure on $X$, which is $G$-invariant. Then
both systems $(X,\mathcal{B}, \mu, G)$ and $(X,\mathcal{B}, \mu, \Gamma)$ are prime.
\end{enumerate}
Moreover if $\Gamma$ is Gromov hyperbolic, then both $C^*_r (\Gamma) \subset C(X) \rtimes_r \Gamma$ and $L(\Gamma) \subset L^{\infty}(X,\mu) \rtimes \Gamma$ are minimal ambient. 
\end{theorem}
\begin{proof}
(1) The transitive system $(X,G)$ is prime because $L$ is a maximal closed subgroup.
If $\pi : (X, \Gamma) \to (Y, \Gamma)$ is a factor map, then it is determined by the $\Gamma$-invariant,
closed equivalence relation $R =\{(x,x'): \pi(x) = \pi(x')\}$. However, because $\Gamma$ is dense in $G$ 
and $R$ is a closed set; it follows that
$R$ is also $G$-invariant, so that it defines also a $G$-factor map.
As $(X, G)$ is prime, this implies that either $R = X \times X$ or $R = \Delta_X$.

(2) Suppose that $\pi : (X,\mathcal{B},\mu, \Gamma)
\to (Y, \mathcal{D}, \nu, \Gamma)$ is a factor map. This means that 
$\mathcal{D}$ is a $\Gamma$-invariant sub-$\sigma$-algebra
of $\mathcal{B}$.
If $\Gamma \ni \gamma_i \to g \in G$, then by definition of the weak topology on $\gamma_i A$ converges to the set
$gA$ in the measure algebra $(\mathcal{B},\mu)$. 
This shows that the $\sigma$-algebra $\mathcal{D}$
is also $G$-invariant and, as the system 
$(X,\mathcal{B},\mu, G)$ is prime, so that
$\mathcal{D}$ is either all of $\mathcal{B}$, or just the trivial
$\sigma$-algebra $\{X, \emptyset\}$.
Thus, the $\Gamma$ system
$(X,\mathcal{B},\mu, \Gamma)$ is also prime.

The topology induced on $\Gamma$ as a subgroup of either $\Homeo(X)$ or $\Aut(X,\cB,\mu)$ coincides with the topology inherited from $G$. Thus,
any sequence of non-identity elements in $\Gamma$ converging in $G$ to the 
identity element is rigid. Thus the flow $(X,\Gamma)$
is uniformly rigid, measure rigid respectively, hence reflecting by Theorem \ref{thm:main} in both cases. 
Now Proposition \ref{am} implies a minimal ambient inclusion. 
\end{proof}
\begin{remark}
In the above theorem, it is enough to assume that $\Gamma$ is a convergence group with property (AP).
\end{remark}

In \cite[Main Theorem]{suzuki2017minimal}, Suzuki presented the first example of a 
minimal ambient inclusion
$C_r^*(\Gamma)\subset C(X)\rtimes_r\Gamma$ such that $C_r^*(\Gamma)$ is non-nuclear, while $C(X)\rtimes_r\Gamma$ is nuclear. Using Theorem \ref{thm:minimal_ambient} enables us to construct additional such examples:
\begin{example}\label{ex-nuclear}
Let $k$ be a local field and $\Gamma < G = \PSL(2,k)$ be a countable, dense, free subgroup. Let $(\bP^1,\Gamma)$ be the standard action on the projective plane $\bP^1=\bP^{1}(k^2)$. Then $C_r^*(\Gamma)\subset C(\bP^1)\rtimes_r\Gamma$ is a minimal ambient inclusion with $C(\bP^1)\rtimes_r\Gamma$ nuclear while $C_r^*(\Gamma)$ is non-nuclear. 
\end{example}
\begin{proof}
Recall that $\bP^1 = G/P$ with $P < G$ is the parabolic group of upper triangular matrices. The minimal ambiance of the inclusion follows directly from Theorem \ref{thm:minimal_ambient}. It is well known that this action is topologically amenable; hence $C(X)\rtimes_r\Gamma$ is nuclear by ~\cite[Theorem~4.3.4]{BO:book}. Finally since $\Gamma$ is a $C^*$-simple group, $C_r^*(\Gamma)$ is non-nuclear.
\end{proof}

\begin{example} \label{ex-lat}
Let $G = \SL_d(\R)$, $\Sigma < G$
a maximal lattice and  $\Gamma < G$ a dense free subgroup. Then the inclusion 
 $L(\Gamma)\subset L^{\infty}(G/\Sigma,\mu) \rtimes \Gamma$ 
 is minimal ambient, where $\mu$ is the $G$ invariant measure on $X = G/\Sigma$.
 \end{example}

\begin{proof}
As a maximal lattice, $\Sigma < G$ is a maximal closed subgroup. Indeed, let $\Sigma \lneqq L < G$ be any closed supergroup. Since $\Sigma$ was chosen as a maximal lattice, $L$ cannot be discrete. Nevertheless, $L$ is normalized by the lattice, which is Zariski dense by Borel's density theorem. So $L \lhd G$ and as $G$ is simple, this means that $L=G$. Now the result follows directly from Theorem \ref{thm:minimal_ambient} 
 \end{proof}

 Similarly, we get the topological examples when $d =2$.

\begin{example} \label{ex-lat-top}
Let $G = \SL_2(\R)$, $\Sigma < G$
a maximal cocompact lattice and $\Gamma < G$ a dense free subgroup. Then the inclusion 
 $C^*_r(\Gamma)\subset C(X) \rtimes_r \Gamma$ 
 is minimal ambient, where $X = G/\Sigma$.
 \end{example}

\begin{remark}
In the setting of Theorem \ref{thm:minimal_ambient} and in Examples \ref{ex-nuclear} \ref{ex-lat} and \ref{ex-lat-top}, one needs a dense convergence group with property (AP) to apply the theorem. The existence of dense free subgroups is well-known for many topological groups. See, for example, \cite{Epstein:free_subgroups}, \cite{GN:ubiquity}, \cite{BG:dense_free} and the references therein. To get a richer source of examples of dynamical systems, it is desirable to 
% have a good source for 
find more general types of dense convergence groups. 

Interestingly, in many cases, the existence of a dense free group already guarantees the existence of dense surface groups and even a dense copy of every limit group in the sense of Sela \cite{Sela:Diophantine_geom_over_free_groups}. It is conjectured in \cite{BG:limit} (see also \cite{BGSS:surface}) that every locally compact group containing a dense copy of $F_2$ contains a dense copy of every limit group. This 
is proved in the same paper whenever $G$ is an algebraic group over a local field of characteristic zero. Thus, one can replace the dense free group in Example \ref{ex-nuclear} with any other limit group. 
\end{remark}

\bibliographystyle{amsalpha}
\bibliography{rigidactions}

\end{document}